\documentclass[a4paper]{amsart}

\usepackage{amsmath}
\usepackage{amssymb}
\usepackage{amssymb}
\usepackage{latexsym}
\usepackage{enumerate}
\usepackage{amsthm}
\newtheorem{theoremA}{Theorem}

\renewcommand{\thetheoremName}

\newtheorem*{thmain}{Main Theorem}

\newtheorem{proposition[[]]}[theoremName]{Proposition G}

\newtheorem{theorem}{Theorem}[section]
\newtheorem{lemma}[theorem]{Lemma}

\newtheorem{proposition}[theorem]{Proposition}

\theoremstyle{definition}
\newtheorem{definition}[theorem]{Definition}

\newtheorem{remark}{Remark}

\numberwithin{equation}{section}

\newcommand{\qedd}{\hfill \Box}

\newcommand{\Vol}{\operatorname{Vol}}

\newcommand{\kan}{\mathbb{K}^{n}(\kappa)}

\newcommand{\erre}{\mathbb{R}}

\newcommand{\K}{\kappa}
\newcommand{\Ct}{\text{Ct}_\kappa}
\newcommand{\Q}{\mathcal{Q}}
\newcommand{\Rd}{\frac{R}{2}}

\newcommand{\Han}{\mathbb{H}^n(\kappa)}
\newcommand{\Ham}{\mathbb{H}^m(\kappa)}

\newcommand{\grad}{\operatorname{\nabla}}

\newcommand{\Ss}{\Vol(S_s^{\kappa})}
\newcommand{\Sr}{\Vol(S_R^{\kappa})}

\newcommand{\Bs}{\Vol(B_s^{\kappa})}
\newcommand{\Br}{\Vol(B_R^{\kappa})}

\begin{document}

\title[On the  fundamental tone of  minimal submanifolds]{On the  fundamental tone of  minimal submanifolds with controlled  extrinsic curvature
}


\author{Vicent Gimeno         
}

\address{Department of Mathematics, Universitat Jaume I, Castell\'o de la Plana, Spain                        
}
\email{vigigar@postal.uv.es}


\begin{abstract}
The aim of this paper is to obtain the fundamental tone for minimal submanifolds of the Euclidean or hyperbolic space under certain restrictions on the extrinsic curvature. We show some sufficient conditions on the norm of the second fundamental form that allow us to obtain the same upper and lower bound for the fundamental tone of minimal submanifolds in a Cartan-Hadamard ambient manifold. As an intrinsic result, we obtain a sufficient condition on the volume growth of a Cartan-Hadamard manifold to achieve the lowest bound for the fundamental tone given by McKean.   
\end{abstract}

\maketitle

\section{Introduction}\label{intro}
Let  $(M,g)$ be a Riemannian manifold, the \emph{fundamental tone}
$\lambda^{\ast}(M)$ of  $M$ is
defined by 

\begin{equation}\label{Ray} 
\lambda^{\ast}(M)=\inf\{\frac{\smallint_{M}\vert
\grad f \vert^{2}}{\smallint_{M} f^{2}};\, f\in
 {L^{2}_{1,0}(M )\setminus\{0\}}\}
\end{equation}
\noindent where $L^{2}_{1,\,0}(M )$ is the completion  of
$C^{\infty}_{0}(M )$ with respect to the norm $ \Vert\varphi
\Vert^2=\int_{M}\varphi^{2}+\int_{M} \vert\nabla \varphi\vert^2$.

The fundamental tone is a powerful tool in the challenging area of geometric analysis, being  a classic feature its relation with the curvature of the manifold. We refer to the
book of I. Chavel \cite{Chavel} for a discussion of these and
related concepts. 

For non-positively curved manifolds, lower bounds for the fundamental tone are well-known. For these cases, McKean gave the following bounds for the fundamental tone of a Cartan-Hadamard manifold.
\begin{theorem}[See \cite{McKean}]\label{McKeanTh}Let $M^n$ be a complete, simply connected manifold with sectional curvature bounded above by $\kappa<0$. Then
\begin{equation}\label{McKean}
\lambda^*(M)\geq \frac{-(n-1)^2\kappa}{4}\quad .
\end{equation}
\end{theorem}

Similary, L.P. Cheung and L.F. Leung showed that minimal submanifolds of the hyperbolic space  retain this lower bound of the ambient manifold.
\begin{theorem}[See \cite{Cheung-Leung}]\label{Che-Leu}
Let $M^n$ be a complete  minimal submanifold  in the hyperbolic space $\Ham$ of constant sectional curvature $\kappa<0$. Then 
\begin{equation}\label{Cheung}
\lambda^*(M)\geq \frac{-\left(n-1\right)^2\kappa}{4}\quad .
\end{equation}
\end{theorem}

Observe that inequality (\ref{McKean}) and inequality (\ref{Cheung}) are sharp in the sense that when $M$ is the hyperbolic space $\Han$ (considered as a manifold or considered as a totally geodesic submanifold of $\Ham$) the fundamental tone $\lambda^*(M)$ satisfies   $\lambda^*(M)=\frac{-(n-1)^2\kappa}{4}$ (see \cite{Chavel} or \cite{McKean}). 

Moreover, it is known that  other minimal submanifolds exist as well as the hyperbolic space with that lowest fundamental tone. For example, all the classical minimal catenoids in the hyperbolic space $\mathbb{H}^3(-1)$  given by M. do Carmo and M. Dajczer in \cite{DoDa} achieve equality in (\ref{Cheung}) as shown by  A. Candel.
\begin{theorem}[See \cite{Candel}]\label{candel2}
The fundamental tone of the minimal catenoids (given in \cite{DoDa}) in the hyperbolic space $\mathbb{H}^3(-1)$ is 
$$
\lambda^*(M)=\frac{1}{4}\quad .
$$  
\end{theorem}

The purpose of this paper is to provide geometric conditions under which a submanifold (or a manifold) attains the equality in the inequality  (\ref{Cheung}) (or inequality (\ref{McKean}), respectively). Let us emphasize that the spherical catenoids given in the previous example, as it was analyzed  in  \cite{Seo},  have finite  $L^2$-norm of the second fundamental form (denoted by $A$ throughout this paper)
\begin{equation}
\int_M\vert A \vert^2 d\mu<\infty\quad.
\end{equation}

This could suggest that the finiteness of the $L^2$-norm of the second fundamental form for minimal surfaces in the hyperbolic space implies the lowest fundamental tone. This is in fact true, and is the statement of our first theorem.
\begin{theoremA}
Let $M^2$ be a minimal surface immersed in the hyperbolic space $\Ham$ of constant sectional curvature $\kappa<0$. Suppose that $M^2$ has  finite total extrinsic curvature, i.e,  $\int_M\vert A \vert^2 d\mu<\infty$. Then, the fundamental tone satisfies
\begin{equation}
\lambda^*(M)= \frac{-\kappa}{4}\quad .
\end{equation}
\end{theoremA}

 Our approach to the problem makes use of our previous results about
 the influence of the extrinsic curvature restrictions on the volume
 growth of the submanifold (see
 \cite{GP},\cite{GPCheeger},\cite{GPGap} or \S \ref{prelim}), and also the relation between the finite volume growth and the fundamental tone as a new tool (see the Main Theorem).

For higher dimensional minimal submanifolds of the hyperbolic space, under an appropriate decay of the norm of the second fundamental form  on the extrinsic distance  (see \S \ref{prelim} for a precise definition) we obtain 

\begin{theoremA}
Let $M^n$ be a minimal  submanifold properly immersed in the hyperbolic space $\Ham$ of constant sectional curvature $\kappa<0$. Suppose moreover that   $n>2$ and  the submanifold is of faster than exponential decay of its extrinsic curvature. Namely,  there exists a point $p\in M$ such that  
$$
\vert A\vert_x \leq \frac{\delta(r_p(x))}{e^{2\sqrt{-\kappa}r_p(x)}} \quad ,
$$ 
where $\delta(r)$ is a function such that $\delta(r)\to 0$ when $r\to \infty$ and $r_p$ is the extrinsic distance function.
Then, the fundamental tone satisfies

\begin{equation}
\lambda^*(M)= \frac{-(n-1)^2\kappa}{4}\quad .
\end{equation}   
\end{theoremA}

 We also prove for  minimal submanifolds of the Euclidean space the following well known result (see for example \cite{Pac2})

\begin{theoremA} 
Let $M^n$ be a minimal submanifold immersed in the Euclidean space $\erre^m$ with finite total scalar curvature, i.e,  $\int_M\vert A\vert^n d\mu<\infty$. Then
\begin{equation}
\lambda^*(M)=0\quad .
\end{equation}   
\end{theoremA}

\begin{remark}
As far as we know, the question 1.5 given in \cite{Pac2} about the existence (or not) of complete minimal submanifolds properly immersed in the $\erre^m$ with positive fundamental tone is still an open question. In any case, if such kind of immersion exists, from our Main Theorem the immersion should be of infinite volume growth, and from Theorem \ref{two-sides} the submanifold has no an extrinsic doubling property.   
\end{remark}

Finally, as an intrinsic version of our Main Theorem one may  obtain the following theorem in the direction of the McKean's Theorem \ref{McKeanTh}

\begin{theoremA}
Let $M^n$ be a complete, simply connected manifold with sectional curvature bounded from above by $\kappa<0$. Suppose moreover that there exists a point $p\in M$ such that
\begin{equation}\label{intrinsic-condition}
\sup_{R>0}\frac{\Vol(B_R^M(p))}{\Br}<\infty\quad,
\end{equation}
where $B_R^M(p)$ is the geodesic ball in $M$  centered at $p$ of radius $R$, and $B_R^{\kappa}$ is the geodesic ball in $\Han$ of the same radius $R$ .Then
\begin{equation}
\lambda^*(M)= \frac{-(n-1)^2\kappa}{4}\quad .
\end{equation}
\end{theoremA}
 
\begin{remark}
In view of the intrinsic version of the results of \cite{Esteve-Palmer} for complete, simply connected, $2$-dimensional manifold $M^2$ with Gaussian curvature $K_M$ bounded from above by $K_M\leq \kappa<0$, the condition (\ref{intrinsic-condition}) can be achieved provided the following 
\begin{equation}
\int_M(\kappa-K_M)d\mu<\infty\quad .
\end{equation}
In particular (in view of \cite{Esteve-Palmer}), for every   $2$-dimensional Cartan-Hadamard manifold $M$ which is asymptotically locally $\kappa$-hyperbolic of order $2$ (see \cite{Shi},\cite{Esteve-Palmer}) and with sectional
curvatures bounded from above by  $K_M \leq \kappa < 0$, the fundamental tone satisfies
$$
\lambda^*(M)=\frac{-\kappa}{4}\quad.
$$
\end{remark}

 The proof of the three previous extrinsic theorems (Theorem A, Theorem B and Theorem C) uses the fact that under the hypothesis of the theorems the submanifold has finite volume growth ($\displaystyle \lim_{R\to \infty}\Q(R)<\infty$, see \S \ref{prelim}, and \S \ref{finite-growth} for precise definitions and see also \cite{Tkachev} for an alternative definition in terms of projective volume). Hence, the theorems will be proved using the following Main Theorem:


\begin{thmain}\label{MainTh}
Let $M^n$ be a $n-$dimensional minimal submanifold properly immersed in a simply connected Cartan-Hadamard manifold $N$ of sectional curvature $K_N$ bounded from above by  $K_N\leq \kappa\leq
0$. Suppose that
\begin{equation}
\lim_{R\to \infty}\Q(R)<\infty\quad.
\end{equation}
Then,
\begin{equation}\label{equality}
\lambda^*(M)=-\frac{(m-1)^2\K}{4}\quad .
\end{equation}
\end{thmain}

To obtain the Main Theorem we estimate upper bounds for the
fundamental tone.  According to S.T. Yau (\cite{YauReview}) it is  important to find upper bounds for the fundamental tone.
In the case of stable minimal submanifolds of the hyperbolic space, A. Candel
gave in 2007  the following upper bound for the fundamental tone
\begin{theorem} [See \cite{Candel}]\label{candel1}
Let $M$ be a complete simply connected stable minimal surface in the $3-$dimensional hyperbolic space $\mathbb{H}^3(-1)$. Then the fundamental tone of $M$ satisfies
\begin{equation}
\frac{1}{4}\leq \lambda^*(M)\leq \frac{3}{4}\quad .
\end{equation}
\end{theorem}

Using only as  hypothesis the finiteness of the $L^2-$norm of the second fundamental form, K. Seo has recently generalized  the result of Theorem \ref{candel1} without using the hypothesis about the simply connectedness

\begin{theorem} [See \cite{Seo}]
Let $M^n$ be a complete stable minimal hypersurface in $\mathbb{H}^{n+1}(-1)$ with $\int_M\vert A\vert^2d\mu<\infty$. Then we have
\begin{equation}
\frac{\left(n-1\right)^2}{4}\leq \lambda^*(M)\leq n^2\quad .
\end{equation}
\end{theorem}

Another achievement of this paper is that using only the finiteness of the $L^2-$norm of the second fundamental form (Theorem A), or an appropriate decay of the norm of the second fundamental form, we can also remove the hypothesis about the stability, and obtain not only an inequality but an equality on the fundamental tone.

\begin{remark}
Note that the finiteness of the $L^2$-norm of the second fundamental form of a minimal surface in the hyperbolic space does not imply the stability of the surface. In fact, M. do Carmo and M. Dajczer also proved in \cite{DoDa} that there exist some unstable spherical  catenoids in $\mathbb{H}^3(-1)$. And observe, moreover, that the codimension of the submanifold plays no role in our theorems, in contrast to what happens in the theorems from A. Candel and K. Seo where the codimension must be $1$.
\end{remark}

The structure of the paper is as follows.

In \S \ref{prelim} we recall the definition of the extrinsic distance
function, the extrinsic ball and the volume growth function
$\Q$.  Showing that under the hypothesis of Theorems A, B, and C the
submanifold has finite volume growth
\begin{equation}
\sup_{R>0}\Q(R)<\infty\quad .
\end{equation}  

In \S \ref{main-prop} we will prove the Main
Theorem.  Finally as a corollary from the proof of the Main Theorem, supposing that the submanifold has an extrinsic doubling property we state the Theorem \ref{two-sides} where is obtained lower and upper bounds for the fundamental tone of a minimal submanifold properly immersed in a Cartan-Hadamard ambient manifold.

 \section{Preliminaries}\label{prelim}
The proof of Theorems A, B and C is based on upper and lower bounds for the fundamental tone. The lower bounds are well known (see Theorem \ref{Che-Leu}), but to obtain upper bounds we use the so called \emph{volume growth function} and its relation to the behavior of the extrinsic curvature. The volume growth function is the quotient between the volume of an extrinsic ball of radius $R$ and the geodesic ball of the same radius $R$ in an appropriate real space form.

Let $\kan$ denote the $n-$dimensional simply connected real space form of constant sectional curvature $\K \leq 0$, recall that  the volume of the  geodesic sphere $S^\K_R$ and the geodesic ball $B_R^\K$ of radius $R$ in $\kan$ are (see \cite{Tubes}) 
$$
\Sr=\omega_n\text{S}_\K(R)^{n-1}\quad \Br=\int_0^R \Ss ds\quad ,
$$
being $\omega_n$ the volume of the geodesic sphere of radius $1$ in $\erre^n$, where $\text{S}_\kappa$ is the usual function

$$
\text{S}_\K(t)=\begin{cases}
t\quad \text{if}\quad \K=0,\\
\frac{\sinh(\sqrt{-\K}t)}{\sqrt{-\K}}\quad\text{if}\quad \K<0\quad .
\end{cases} 
$$

\noindent And the mean curvature pointing inward of the geodesic spheres of radius $R$  is  

$$
\Ct (t)= \frac{\text{S}_\kappa'(t)}{\text{S}_\kappa(t)}\quad.
$$

On the other hand, given a submanifold $M$ immersed in a simply connected Cartan-Hadamard manifold  $N$ of sectional curvature $K_N$ bounded from a above by $K_N\leq \kappa\leq 0$, the extrinsic ball $M_p^R$ centered at $p\in M$ of radius $R$ is the sublevel set of the extrinsic distance function $r_p$, where the extrinsic distance function is
\begin{definition}[Extrinsic distance function] Let $\varphi: M\to N$  be an immersion from the manifold $M$ to the simply connected Cartan-Hadamard manifold of  sectional curvature $K_N$ bounded from above by $K_N\leq \kappa\leq 0$. Given a point $p\in M$, the extrinsic distance function $r_p:M\to \erre^+$ is defined by
$$
r_p(x)=\text{dist}^{N}(\varphi(p),\varphi(x))\quad ,
$$ 
where $\text{dist}^{N}$ denotes the usual geodesic distance function in $N$.
\end{definition}
Therefore the extrinsic ball $M_p^R$ centered at $p\in M$ of radius $R$ is
$$
M_p^R:=\left\{x\in M\,:\,r_p(x)<R\right\}\quad.
$$

With the extrinsic ball and the geodesic ball we can define the volume growth function

\begin{definition} [Volume growth function] Let $\varphi: M^n\to N$  be an immersion from the $n-$dimensional manifold $M$ to the simply connected Cartan-Hadamard manifold of sectional curvature $K_N$ bounded from above by $K_N\leq \kappa\leq 0$. The volume growth function $\Q: \erre^+ \to \erre^+$ is given by 
$$
\Q(R):=\frac{\Vol(M_p^R)}{\Br}\quad .
$$
Where $\Br$ is the geodesic ball of radius $R$ in $\kan$.
\end{definition}

\subsection{Volume growth  of minimal submanifolds of controlled extrinsic curvature}\label{finite-growth}

Firstly  to  prove the Theorems A, B and C we have to  study the behavior of  the volume comparison function for  minimal submanifolds of a Cartan-Hadamard manifold. Let us recall the following monotonicity formula

\begin{proposition}[Monotonicity formula, see \cite{Palmer}]\label{Palmer}
Let $\varphi: M^n\to N$  be an immersion from the $n-$dimensional manifold $M$ to the simply connected Cartan-Hadamard manifold of sectional curvature $K_N$ bounded from above by $K_N\leq \kappa\leq 0$. Then, the volume growth function is a non-decreasing function.
\end{proposition}

Secondly to prove Theorems A, B, and C we need  to
check that an appropriate control of their second fundamental form
allow us to apply the Main Theorem (\S \ref{main-prop}). In order
to apply the Main Theorem the submanifold should have finite volume growth
\begin{equation}
\lim_{R\to\infty}\Q(R)<\infty\quad.
\end{equation}

But we can apply the Main Theorem under the assumptions of the Theorems A, B and C because of the following three theorems that we recall here:

\begin{theorem}[See \cite{A1},\cite{ChSubv} and \cite{GP}]
Let $M^n$ be a minimal submanifold immersed in the Euclidean space $\erre^m$. If $M^n$ has finite total scalar curvature 

$$\int_M\vert A\vert^n d\mu<\infty\quad.$$
Then
$$
\sup_{R>0}\Q(R)<\infty\quad .
$$ 

\end{theorem}

\begin{theorem}[See \cite{ChOss}, \cite{GP} and \cite{Che3}]
Let $M^2$ be a minimal surface immersed in the hyperbolic space $\Ham$ of constant sectional curvature $\kappa<0$ or in the Euclidean space $\erre^m$. If $M$ has  finite total extrinsic curvature, namely  $\int_M\vert A \vert^2 d\mu<\infty$, then $M$ has finite topological type, and
$$
\sup_{R>0} \Q(R)\leq \frac{1}{4} \int_M\vert A \vert^2 d\mu+\chi(M)\quad ,
$$
being $\chi(M)$ the Euler characteristic of $M$.
\end{theorem}

\noindent And
 
\begin{theorem}[see \cite{GPGap}]
Let $M^n$ be a minimal $n-$dimensional submanifold properly immersed in the hyperbolic space $\Ham$ of constant sectional curvature $\kappa<0$.  If $n>2$ and  the submanifold is of faster than exponential decay of its extrinsic curvature,  namely,  there exists a point $p\in M$ such that  
$$
\vert A\vert_x \leq \frac{\delta(r_p(x))}{e^{2\sqrt{-\kappa}r_p(x)}} \quad ,
$$ 
where $\delta(r)$ is a function such that $\delta(r)\to 0$ when $r\to \infty$. Then the submanifold has finite topological type, and
$$
\sup_{R>0} \Q(R)\leq \mathcal{E}(M)\quad ,
$$
being $\mathcal{E}(M)$ the (finite) number of ends of $M$.
\end{theorem}

\section{Proof of the Main Theorem: volume growth behavior and fundamental tone}\label{main-prop}
The way to proof the equality (\ref{equality}) in the Main Theorem is to obtain the same upper and lower bound for the fundamental tone.
Hence, first of all, we need lower bounds for the fundamental tone. But the lower bounds are well known, and it is straight forward in a similar way to Theorem  \ref{Che-Leu} that
$$
\lambda^*(M)\geq \frac{-(n-1)^2\K}{4}\quad .
$$

\noindent The above lower bound  can also be  proved  using the Cheeger isoperimetric constant of the submanifold. Taking into account that  the Cheeger constant $h(M)$ of a minimal submanifold $M^n$ properly immersed in a Cartan-Hadamard manifold $N$ of sectional curvatures $K_N$ bounded from above by $K_N\leq \K\leq 0$ is (see \cite{GPCheeger})
\begin{equation}
h(M)\geq (n-1)\sqrt{-\K}\quad,
\end{equation}

\noindent therefore,
\begin{equation}
\lambda^*(M)\geq\frac{h(M)^2}{4}\geq  \frac{-(n-1)^2\K}{4}\quad .
\end{equation}

To obtain the upper bounds for the fundamental tone, we use the Rayleigh quotient definition (\ref{Ray}) with an appropriate testing function.
The first step to obtain the testing function is to  define  the real function $\phi: \erre \to \erre$, given  by

\begin{equation}\label{testing}
\phi(t)=
\begin{cases} 
f(t) \quad\text { if  } t\in [\frac{R}{2},R]\quad,\\
0 \quad \text{ otherwise}	
\end{cases}
\end{equation}

where the function $f:\erre \to\erre$ is
\begin{equation}
f(s)=\frac{\sin(\frac{2\pi(s-\frac{R}{2})}{R})}{\Ss^{1/2}}\quad.
\end{equation}

Now, the second step to take to construct the testing function $\Phi$, is to transplant $\phi$ to $M$ using the extrinsic distance function by the following definition:
$$
\Phi:M\to \erre;\quad \Phi(x)=\phi(r_p(x))\quad.
$$  

By the  Rayleigh quotient definition and the coarea formula
\begin{equation}\label{ray-eq}
\begin{aligned}
\lambda^*(M)\leq &\frac{\int_M \langle \nabla \Phi,\nabla \Phi\rangle d\mu }{\int_M \Phi^2 d\mu}=\frac{\int_M (\phi')^2\langle \nabla r_p,\nabla r_p\rangle d\mu }{\int_M \phi^2 d\mu}\leq \frac{\int_M (\phi')^2  d\mu }{\int_M \phi^2 d\mu}\\
=&\frac{\int_0^R \left[\int_{\partial M_p^s}\frac{(\phi')^2}{\vert \nabla r\vert}  \right]ds}{\int_0^R \left[\int_{\partial M_p^s}\frac{\phi^2}{\vert \nabla r\vert}  \right]ds}=\frac{\int_{\Rd}^R (\phi'(s))^2\left[\int_{\partial M_p^s}\frac{1}{\vert \nabla r\vert}  \right]ds}{\int_{\Rd}^R \phi^2(s)\left[\int_{\partial M_p^s}\frac{1}{\vert \nabla r\vert}  \right]ds}\\
=&\frac{\int_{\Rd}^R (\phi'(s))^2\left(\Vol(M_p^s)\right)'ds}{\int_{\Rd}^R \phi^2(s)\left(\Vol(M_p^s)\right)'ds}\quad.
\end{aligned}
\end{equation}

Using the following two lemmas
\begin{lemma}\label{volume-comparison}
\begin{equation}
\Q(s)\Ss\leq \left(\Vol(M_p^s)\right)'= \left(\ln\Q(s) \right)'\Bs\Q(s)+\Q(s) \Ss\quad .
\end{equation}
\end{lemma}
\begin{proof}
From the definition of $\Q$ and taking into account that $\Q$ is a non-decreasing function (by proposition \ref{Palmer})
\begin{equation}
\left(\ln \Q(s)\right)'=\frac{\left(\Vol M_p^s\right)'}{\left(\Vol M_p^s\right)}-\frac{\Ss}{\Bs}\geq 0\quad .
\end{equation}
So, 
\begin{equation}
\Q(s)\Ss\leq \left(\Vol(M_p^s)\right)'=\left(\ln\Q(s) \right)'\Bs\Q(s)+\Q(s) \Ss\quad .
\end{equation}
$\qedd$
\end{proof}

\begin{lemma}\label{testing-prop}
There exists an upper bound function $\Lambda: \erre^+\to \erre^+$
such that:
\begin{equation}
\frac{\int_0^R (\phi')^2 \Ss ds}{\int_0^R \phi^2 \Ss ds}\leq \Lambda(R)\quad,
\end{equation}
and
\begin{equation}
\lim_{R\to \infty}\Lambda(R)=\frac{-(m-1)\kappa}{4}
\end{equation}
\end{lemma}

\begin{proof}

\begin{equation}
\begin{aligned}
\frac{\int_0^R (\phi'(s))^2 \Ss ds}{\int_0^R \phi(s)^2 \Ss
  ds}=\frac{\int_{\Rd}^R (f'(s))^2 \Ss ds}{\int_{\Rd}^R f(s)^2 \Ss
  ds}\quad .
\end{aligned}
\end{equation}

But
\begin{equation}
\begin{aligned}
\left(f'(s)\right)^2=&\frac{\left(-\frac{m-1}{2}\Ct (s)\sin(\frac{2\pi(s-\Rd)}{R})+\frac{2\pi}{R}\cos(\frac{2\pi(s-\Rd)}{R})\right)^2}{\Ss}\\
\leq &
\frac{\frac{(m-1)^2}{4}\Ct(s)^2\sin(\frac{2\pi(s-\Rd)}{R})^2+\frac{4\pi^2}{R^2}+\frac{2(m-1)\pi}{R}\Ct
(s) }{\Ss}\quad.
\end{aligned}
\end{equation}
And,  since $\Ct(t)$ is a non-increasing function and $ \int_{\Rd}^R \sin(\frac{2\pi(s-\Rd)}{R})^2 ds=\frac{R}{4}$
\begin{equation}
\begin{aligned}
\frac{\int_0^R (\phi'(s))^2 \Ss ds}{\int_0^R \phi(s)^2 \Ss ds}
\leq & 
\frac{(m-1)^2}{4}\Ct(R/2)^2+\frac{8\pi^2}{R^2}\\&+\frac{4\pi(m-1)}{R}\Ct(R/2)\quad. 
\end{aligned}
\end{equation}
Then, letting 
\begin{equation}
\Lambda(R):=\frac{(m-1)^2}{4}\Ct(R/2)^2+\frac{8\pi^2}{R^2}+\frac{4\pi(m-1)}{R}\Ct(R/2)\quad,
\end{equation}
and taking into account that
\begin{equation}
\lim_{R\to\infty}\Ct (R/2)=\sqrt{-\kappa}\quad,
\end{equation}
the lemma is proven. $\qedd$
\end{proof}

\noindent Denoting now,  $F(R):=\left(\frac{(m-1)^2}{4}\Ct(R/2)^2+\frac{4\pi^2}{R^2}+\frac{2(m-1)\pi}{R}\Ct(R/2)\right)$ and $\delta(R):=\int_{\Rd}^R \left(\ln\Q(s) \right)' ds$, applying the lemma \ref{volume-comparison} and the lemma \ref{testing-prop} to inequality (\ref{ray-eq}) we get since $\frac{\Bs}{\Ss}$ is a non-decreasing function

\begin{equation*}
\begin{aligned}
\lambda^*(M)\leq  &\frac{\Q(R)}{\Q(\Rd)}\frac{\int_{\Rd}^R (\phi'(s))^2\left(\ln\Q(s) \right)'\Bs ds+\int_{\Rd}^R (\phi'(s))^2\Ss ds}{\int_{\Rd}^R \phi^2(s)\Ss ds}\\
\leq& \frac{\Q(R)}{\Q(\Rd)}\left(\frac{4}{R}\int_{\Rd}^R (\phi'(s))^2\left(\ln\Q(s) \right)'\Bs ds+  \Lambda(R)\right)\\
\leq &\frac{\Q(R)}{\Q(\Rd)}\left[\frac{\Br}{\Sr} \frac{4}{R}F(R)\delta(R)+ \Lambda(R)\right]
\end{aligned}
\end{equation*}
Letting $R$ tend to infinity and taking into account that
\begin{equation}
\begin{aligned}
\lim_{R\to\infty}F(R)=&-\frac{(m-1)^2\K}{4}\quad ,\\
\lim_{R\to\infty}\delta(R)=&0 \quad ,\\
\lim_{R\to\infty}\frac{\Br}{\Sr} \frac{4}{R}=&\begin{cases}  
\frac{4}{m-1} \text{ if } \K=0,\\
0\text{  if } b<0.
\end{cases}
\\
\lim_{R\to\infty}\frac{\Q(R)}{\Q(\Rd)}=&1\quad ,\\
\lim_{R\to\infty}\Lambda(R)=&-\frac{(m-1)^2\K}{4}\quad .
\end{aligned}
\end{equation}
\noindent  we conclude the proof of the theorem. 
\begin{remark}
Observe that the finiteness of the volume growth function 
\begin{equation}\label{hypothesis}
\lim_{R\to\infty}\Q(R)<\infty\quad,
\end{equation}
is only used in the proof of the Main Theorem to achieve 
\begin{equation}\label{hypothesis2}
\lim_{R\to\infty}\frac{\Q(R)}{\Q(\Rd)}=1\quad .
\end{equation}
Therefore, we could use the slightly weaker assumption on the volume growth given by the limit  (\ref{hypothesis2}) instead  the assumption on the finite volume growth of the submanifold (limit (\ref{hypothesis})), or use an extrinsic doubling property to obtain two sides estimates for the fundamental tone
\begin{theoremA}\label{two-sides}
Let $M^n$ be a $n-$dimensional minimal submanifold properly immersed in a simply connected Cartan-Hadamard manifold $N$ of sectional curvature $K_N$ bounded from above by  $K_N\leq \kappa\leq
0$. Suppose that the immersion has an extrinsic doubling property, namely
\begin{equation}
\frac{\Q(R)}{\Q(\Rd)}<C\quad.
\end{equation}
Then,
\begin{equation}
-\frac{(m-1)^2\K}{4}\leq \lambda^*(M)\leq-\frac{C(m-1)^2\K}{4}\quad .
\end{equation}
\end{theoremA}
\end{remark}



\providecommand{\noopsort}[1]{}\providecommand{\singleletter}[1]{#1}\def\cprime{$'$} \def\cprime{$'$} \def\cprime{$'$} \def\cprime{$'$}

\end{document}